\documentclass{article}
\usepackage[utf8]{inputenc}
\usepackage[english]{babel}
\usepackage[a4paper,vmargin=0.7in,hmargin=0.7in]{geometry}
\usepackage[misc]{ifsym}
\usepackage{array,booktabs}
\usepackage{amsthm, amssymb, eucal, tikz, amsmath}
\usepackage[colorlinks=true, linkcolor=black, citecolor=black, urlcolor=black]{hyperref}
\usepackage{tikz}
\usetikzlibrary{shapes}
\usetikzlibrary{plotmarks}

\newtheorem{theorem}{Theorem}[section]
\newtheorem{lemma}[theorem]{Lemma}
\newtheorem{corollary}[theorem]{Corollary}
\newtheorem{definition}{Definition}[section]

\newtheorem{proposition}{Proposition}[section]
\newtheorem{observation}{Observation}

\usepackage{authblk}

\title{Distance Magic Labeling of Generalised Mycielskian Graphs}
\author{Ravindra Pawar\footnote{p20200020@goa.bits-pilani.ac.in} \;  and  \; Tarkehswar Singh\footnote{tksingh@goa.bits-pilani.ac.in}\\
	Department of Mathematics, BITS Pilani K K Birla Goa Campus, Goa, India.}
\date{}

\begin{document}
	
\maketitle{}
\begin{abstract}
A graph $G = (V,E)$ is said to be a distance magic graph if there is a bijection $f: V(G) \to \{1,2,\dots, |V(G)|\}$ 
such that the vertex weight $ w(u) = \sum_{v \in N(u)} f(v) = k$ is constant and independent of $u$,
where $N(u)$ is the open neighborhood of the vertex $u$.
The constant $k$ is called a \textit{magic constant}, the function $f$ is called a \textit{distance magic labeling of the graph $G$}
and the graph that admits such labeling is called a \textit{distance magic graph}.  In this paper, we present some results on distance magic labeling of graphs obtained by generalised Mycielskian construction.
		
\end{abstract}

\section{Introduction}
Throughout this paper, by a graph $G = (V, E)$, we mean a finite undirected simple graph with vertex set $V(G)$ and edge set $E(G)$. For graph-theoretic terminology and notation, we refer to \cite{west}.\\
\smallskip
 A \textit{labeling} of a graph is any function that assigns elements of a graph (vertices or edges or both) to the set of numbers (positive integers or elements of groups, etc.). In particular, if we have a bijection $f:V(G) \to \{1, 2,\dots, |V(G)|\}$ then $f$ is called a \textit{vertex labeling}. 
The \textit{neighborhood} of a vertex $x$ in $G$ is the set of all vertices adjacent to $x$ and is denoted by $N_G(x)$. The \textit{degree} of vertex $v$ in $G$, denoted by $deg_{G}(v)$ is the $|N_G(v)|$. \\
\smallskip
When a graph $G$ is clear from the context, we will simply write $N(x)$ and $deg(x)$ for the neighborhood and degree of a vertex $x$, respectively. The \textit{weight} of a vertex $v$, denoted by $w(v)$ is defined as $w(v)= \sum_{u \in N(v)}  f(u)$. 
If $f$ is vertex labeling such that $w(v) = k$, for all $v \in V(G)$, then $k$ is called a \textit{magic constant} and the labeling $f$ is called a \textit{distance magic labeling}. A graph that admits such labeling is called a \textit{distance magic graph}. For more details, see \cite{dm_survey_arumugam, dm_survey_gallian, sigma_jinnah, dm_miller, rao, vilfredt}.\\
\smallskip
The characterization of distance magic graphs remains a persistent challenge within the realm of graph labeling problems, and as of now, no comprehensive characterization has been established. Researchers have been actively investigating this problem, either by examining specific graph families or by constructing graphs with distinct graph-theoretic properties conducive to distance magic labeling. Additionally, there is a prevailing belief that among a collection of nonisomorphic graphs on $n$ vertices, only a small fraction exhibits distance magic properties \cite{dm_algo_fuad}. This paper contributes to the discourse by conducting a study on the distance magic labeling of graphs derived from the generalized Mycielskian construction.\\
\smallskip
The Mycielskian $\mu(G)$ of a graph $G = (V, E)$ is the graph with the vertex set $V(\mu(G)) = \{(x,0), (x,1): x \in V(G)\} \cup \{u\}$ and the edge set $E(\mu(G)) = \{(x,0)(y,0), (x,0)(y,1) : xy \in E(G)\} \cup \{(x,1)u : x \in V(G)\}$\cite{mycielski}. Using this construction, Mycielski proved the existence of triangle-free graphs with arbitrarily large chromatic numbers. In \cite{ravi_myc}, authors have studied the distance magic labeling of Mycielskian of some families of graphs, including trees, cycles, wheels, and complete bipartite graphs. In particular, they observed that, $\mu(G)$ admits distance magic labeling independent of the graph $G$. Furthermore, the authors established that if $G$ is $r$-regular such that $\mu(G)$ is distance magic, then $r \le 3$. They have given the complete characterisation of connected $2$-regular graphs $G$ for which $\mu(G)$ is distance magic. In particular, for a given graph $G$, they have studied the problem for the graphs $G, \mu(G), \mu^{2}(G), \dots $, obtained through the iterative application of Mycielskian construction on $G$.\\
\smallskip
The Mycielski's construction was generalised independently by Stiebitz and Ngoc in their dissertations in the following way: Given a graph $G=(V,E)$ and an integer $m \ge 0$, \textit{Generalised Mycielskian graph} $\mu_m(G)$ as graph with a vertex set $\{(x,i): 0 \leq i \leq m, x\in V(G)\} \cup \{u\}$, where there is an edge $(x,0)(y,0)$ and $(x,i)(y,i+1)$ whenever there is an edge $xy \in E$, and an edge $(x,m)u$ for all $x \in V$. Observe that, for a connected graph $G$ with $|V(G)| = n$, $\mu_m(G)$ is also connected with $|V(\mu_m(G))| = mn + n + 1$. Note that $\mu_0(G) \cong G$ and $\mu_1(G) \cong \mu(G)$.
The construction is illustrated in Figure \ref{fig:m3p3} for $\mu_2(P_3)$.
\smallskip
\begin{figure}[ht]
\centering
\includegraphics[scale = 0.5]{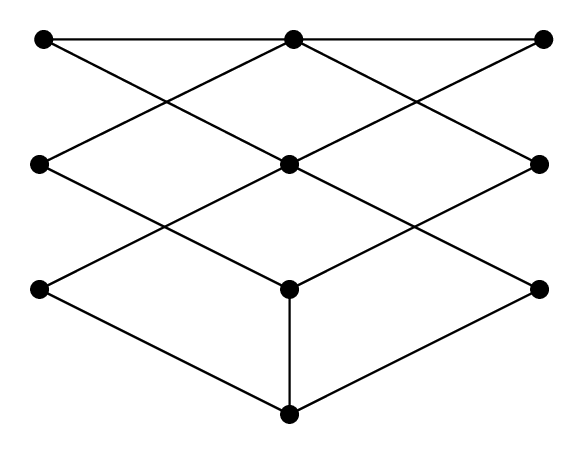}
\caption{$\mu_2(P_3)$.}
\label{fig:m3p3}
\end{figure}

In this paper, we discuss the existence of distance magic labeling of the generalised Mycielskian graph $\mu_m(G)$ for some families of graphs. Throughout the paper, we assume $m \geq 2$.\\
\smallskip
We need the following results for our further investigation.
	
\begin{theorem} \label{VJ} \cite{sigma_jinnah, vilfredt} 
A graph $G$ is not distance magic if there are vertices $x$ and $y$ in $G$ such that $|N(x) \triangle N(y)| =1$ or $2$.
\end{theorem} 
	
\begin{theorem}\cite{dm_miller, rao, vilfredt} \label{weights}
Let $f$ be a distance magic labeling of a graph $G$ with the vertex set $V$. Then, the sum of the weights of all the vertices is given by:
\begin{eqnarray}
\sum_{v \in V(G)} w(v) = \sum_{v \in V(G)} deg(v) f(v) = kn, \nonumber
\end{eqnarray}
where $k$ is magic constant and $n$ is the number of vertices.
\end{theorem}
	
\begin{corollary} \cite{sigma_jinnah, dm_miller, vilfredt}
No odd regular graph is distance magic.
\end{corollary}
	
\begin{theorem}\cite{dm_miller} \label{cycle}
Cycle $C_n$ is distance magic if and only if $n= 4$.
\end{theorem}
\begin{theorem}\cite{dm_miller} \label{wheel}
The wheel $W_n = C_n + K_1$ is distance magic if and only if $n= 4$.
\end{theorem}
\begin{definition}
A function $g : V(G) \to [0, 1]$ is called a \textit{total dominating function} of the graph $G$ if $\sum_{u \in N(v)}g(u) \ge 1$ for all $v \in V$. The \textit{fractional total domination number} of a graph $G$ is denoted by $\gamma_{ft}$ and is given by 
\begin{equation} \nonumber
\gamma_{ft} = \text{ min} \{|g| : g \text{ is total dominating function of } G\},
\end{equation}
where $|g| = \sum_{v \in V} g(v)$.
\end{definition}
\begin{theorem} \cite{mycielski_gen_par} \label{dom_num}
For any graph $G$ without isolated vertices,
\begin{align*}
\gamma_{ft}(\mu_m(G)) = \displaystyle 
\begin{cases}
\frac{ m }{2} \gamma_{ft}(G) + 2 - \frac{1}{\gamma_{ft}(G)} \text{ if } m \text{ is even}\\
\frac{ m+1}{2} \gamma_{ft}(G) + \frac{1}{\gamma_{ft}(G)} \text{ if } m \text{ is odd}.
\end{cases}
\end{align*}
\end{theorem}
\begin{theorem} \cite{uniquek, o_uniquek} \label{unique_k}
If $G$ is a distance magic graph of order $n$ then the  distance magic constant $k = \frac{n(n + 1)}{2 \gamma_{ft}(G)}$.
\end{theorem}
\section{Main Results}
Let $G = (V,E)$ be a graph. The subgraph induced by the subset $S=\{(x,0) \in V(\mu_m(G): x \in V(G)\}$ of $V(\mu_m(G))$ is isomorphic to $G$. By construction, we have $deg_{\mu_m(G)}(x_i,0) = deg_{\mu_m(G)}(x_i,j)$, for all $1 \le j \le (m-1)$.
\begin{observation} \label{obs1}
Let $G$ be a graph. For $x \in V(G)$, $deg_{\mu_m(G)}(x,i) = deg_{\mu_m(G)}(x,0) = 2~deg_{G}(x)$, for $1 \le i \le (m-1)$ and $deg_{\mu_m(G)}(x,m) = \frac{deg_{\mu_m(G)}(x,0)}{2}+1$.    
\end{observation}
	
\begin{proposition}\label{th2}
$\mu_m(G)$ is $r$-regular if and only if $G \cong K_2$.
\end{proposition}
\begin{proof}
Let $G$ be a graph on $n$ vertices such that $\mu_m(G)$ is $r$-regular. Therefore,
\begin{equation}\label{eq:th21}
deg(u) = deg(x,i) = r \text{ for } i = 0,1,2, \dots, m, \forall x \in V(G).  
\end{equation}
But, $deg(u) = n$. Hence, 
\begin{equation}\label{eq:th22}
r = n.
\end{equation}
By Observation \ref{obs1}, we have $deg(x,1) = \frac{deg(x,0)}{2}+1$. Using this in Equation (\ref{eq:th21}) we get $r = 2$. From Equation (\ref{eq:th22}), we get $n = r =2$. This means $G$ is a graph on $2$ vertices such that $\mu_{m}(G)$ is $2$-regular. Now on $2$ vertices, only two graphs are possible: $K_2$ or its complement $\overline{K_2}$. For $x \in V(\overline{K_2})$, $deg_{\mu_m(\overline{K_2})}(x,0) = 0$ and $deg_{\mu_m(\overline{K_2})}(u) = 2$. Therefore, $\mu_m(\overline{K_2})$ is not regular. Hence, $G$ can not be isomorphic to $\overline{K_2}$. Therefore, $G$ must be isomorphic to $K_2$. Conversely if $G \cong K_2$, then $\mu_m(K_2) \cong C_{2r+1}$ which is $2$-regular. This completes the proof.
\end{proof}
Now we provide the sufficient conditions on a graph $G$, for the non-existence of distance magic labeling of its generalised Mycielskian graph $\mu_m(G)$ for $m \geq 2$.
\begin{lemma} \label{deg1}
If the minimum degree of a graph $G$ is $1$ then $\mu_m(G)$ is not distance magic.
\end{lemma}
\begin{proof}
Let $x_1$ be a vertex in $G$ such that $deg_G(x_1) = 1$. Hence, there is a unique vertex $x_2 \in N_G(x_1)$. Then $N_{\mu_m(G)}(x_1, 0) = \{(x_2, 0), (x_2, 1)\}$ and $N_{\mu_m(G)}(x_1, 1) = \{(x_2, 0), (x_2, 2)\}$. This gives $|N_{\mu(G)}(x_1,0) \triangle N_{\mu(G)}(x_1,1)| = |\{(x_2,0), (x_2, 2)\}| = 2$. Hence, by Theorem \ref{VJ}, $\mu_m(G)$ is not distance magic. 
\end{proof}
	
This lemma immediately gives the non-existence of magic labeling of Mycielskian of a major family of graphs.
	
\begin{corollary} \label{cordeg1}
If $T$ is a tree, then $\mu_m(T)$ is not distance magic. 
\end{corollary}
\begin{lemma} \label{VJ for Mycielski}
If a graph $G$ contains two vertices $x_i$ and $x_j$ such that $|N_G(x_i) \triangle N_G(x_j)| = 2$, then $\mu_m(G)$ is not distance magic.
\end{lemma}
\begin{proof}
Let $x_i$ and $x_j$ be vertices in $G$ such that $|N_G(x_i) \triangle N_G(x_j)| = 2$. Then $N_{\mu_m(G)}(x_i,m) = \{(y,m-1):y \in N_G(x_i)\} \cup \{u\}$ and $N_{\mu_m(G)}(x_j,m) = \{(y,m-1):y \in N_G(x_j)\} \cup \{u\}$. Therefore,
\begin{align*}
& |N_{\mu_m(G)}(x_i,m) \triangle N_{\mu_m(G)}(x_j,m)| = |N_G(x_i) \triangle N_G(x_j)| = 2
\end{align*}
and by Theorem \ref{VJ}, $\mu_m(G)$ is not distance magic.
\end{proof}
\begin{corollary} \label{kn}
$\mu_m(K_n)$ is not distance magic graph for any $n \ge 1$.
\end{corollary}
\begin{proof}
Let vertices of $K_n$ be labeled as $x_1$, $x_2$,\dots, $x_n$ in anticlockwise manner. If $n = 1$, then $\mu(K_1) \cong K_1 \cup K_2$ is not distance magic. So, we assume $n \ge 2$. Then, $N_{K_n}(x_1) = \{x_2,\ x_3, \dots, x_n\}$ and $N_{K_n}(x_2) = \{x_1,\ x_3,\ x_4, \dots,\ x_n\}$. Hence, $|N_{K_n}(x_1) \triangle N_{K_n}(x_2)| = 2$ and by Lemma \ref{VJ for Mycielski}, $\mu_m(K_n)$ is not distance magic.
\end{proof}
\begin{corollary} \label{cncorollary}
If $n (\ne 4) \ge 3$ then $\mu_m(C_n)$ is not distance magic.
\end{corollary}
\begin{proof}
Let $x_1, x_2, \dots, x_n$  be the vertices of the cycle $C_n$ taken anticlockwise manner. Now the $|N(x_1) \triangle N(x_{n-1})| = |\{x_2, x_{n-2}\}| = 2$ and hence by Lemma \ref{VJ for Mycielski}, $\mu_m(C_n)$ is not distance magic for $n (\ne 4) \ge 3$.
\end{proof}
We know that $\gamma_{ft}(C_n) = \frac{n}{2}$. By Theorem \ref{dom_num},
\begin{equation*}
\gamma_{ft}(\mu_m(C_n) = 
\begin{cases}
\frac{mn^2+8n-8}{4n} \text{ if } m \text{ is even}\\
\frac{n^2(m+1) + 8}{4n} \text{ if } m \text{ is odd}
\end{cases}
\end{equation*}
and by Theorem \ref{unique_k}, if $\mu_m(C_n)$ is distance magic then its magic constant is given by
\begin{equation*}
k = 
\begin{cases}
\frac{2n(mn + n + 1)(mn + n + 2)}{mn^2 + 8n - 8} \text{ if } m \text{ is even}\\
\frac{2n(mn+n+1)(mn+n+2)}{n^2(m+1) + 8} \text{ if } m \text{ is odd}.
\end{cases}
\end{equation*}
Observe that $k = 8m + 10$ for $n = 4$ and for any $m \ge 1$, and it need not always be an integer for other values of $m$ and $n$. Hence, it is not possible to conclude that $\mu_{m}(G)$ is distance magic by using Theorem \ref{unique_k}. Now, we present detailed proof of the existence of distance magic labeling of generalised Mycielskian of cycles.
\begin{theorem}\label{cn}
The generalized Myscielskian $\mu_m(C_n)$ is distance magic if and only if $n = 4$.
\end{theorem}
\begin{proof}
Let $n \ne 4$, then by Corollary \ref{cncorollary}, $\mu_m(C_n)$ is not distance magic. For $n = 4$, let $C_4 = x_1, x_2, x_3, x_4, x_1$ be the cycle on $4$ vertices. We define the labeling $f$ as follows:
\begin{align*}
f(x_i, j) &= 
\begin{cases}
2j + 1, \text{ if } i = 1, j = 0, 1, \dots, m\\
2j + 2, \text{ if } i = 2, j = 0, 1, \dots, m\\
4(m+1) - 2j, \text{ if } i = 3, j = 0, 1, \dots, m\\
4(m+1) - 2j -1, \text{ if } i = 4, j = 0, 1, \dots, m
\end{cases}
\end{align*}
and $f(u) = 4m + 5$. It is easy to see that $f$ is a distance magic labeling with magic constant $8m+10$. A distance magic labeling of $\mu_3(C_4)$ is given in Figure \ref{fig:my.cycle}.
\end{proof}
\begin{figure}[ht]
\centering
\begin{tikzpicture}
                \draw[fill=black] (0,0) circle (3pt);
			\draw[fill=black] (1.5,0) circle (3pt);
			\draw[fill=black] (3,0) circle (3pt);
			\draw[fill=black] (4.5,0) circle (3pt);
			
			\draw[fill=black] (0,-1.5) circle (3pt);
			\draw[fill=black] (1.5,-1.5) circle (3pt);
			\draw[fill=black] (3,-1.5) circle (3pt);
			\draw[fill=black] (4.5,-1.5) circle (3pt);
			
			\draw[fill=black] (0,-3) circle (3pt);
			\draw[fill=black] (1.5,-3) circle (3pt);
			\draw[fill=black] (3,-3) circle (3pt);
			\draw[fill=black] (4.5,-3) circle (3pt);
			
			\draw[fill=black] (2.25,-5) circle (3pt);
			
			\draw [thick] (0,0) -- (1.5,0) -- (3,0) -- (4.5,0);
			\draw[thick] (0,0) to [bend left=50] (4.5,0);
			\draw[thick] (4.5,0) to [bend right=50] (0,0);
			\draw[thick] (0,0) -- (1.5,-1.5) -- (3,0) -- (4.5,-1.5) -- (0,0);
			\draw[thick] (0,-1.5) -- (1.5,0) -- (3,-1.5) -- (4.5, 0) -- (0,-1.5);
			\draw[thick] (0,-3) -- (1.5,-1.5) -- (3,-3) -- (4.5,-1.5) -- (0,-3);
			\draw[thick] (0,-1.5) -- (1.5,-3) -- (3,-1.5) -- (4.5,-3) -- (0,-1.5);
			\draw[thick] (0,-3) -- (2.25,-5) -- (1.5,-3);
			\draw[thick] (3,-3) -- (2.25,-5) -- (4.5,-3);
			
			\node at (2.25,-5.5) {$13$};
			\node at (-0.5,-3.3) {$5$};
			\node at (1.2,-3.3) {$6$};
			\node at (2.5,-3.3) {$8$};
			\node at (5,-3.3) {$7$};
			\node at (-0.5,0) {$1$};
			\node at (-0.5,-1.5) {$3$};
			\node at (1.5,0.3) {$2$};
			\node at (3,0.3) {$12$};
			\node at (5,0) {$11$};
			\node at (5,-1.5) {$9$};
			\node at (1.2,-1.5) {$4$};
			\node at (2.6,-1.5) {$10$};
		\end{tikzpicture}
		\caption{Distance magic labeling of $\mu_2(C_4)$.}
		\label{fig:my.cycle}
\end{figure}
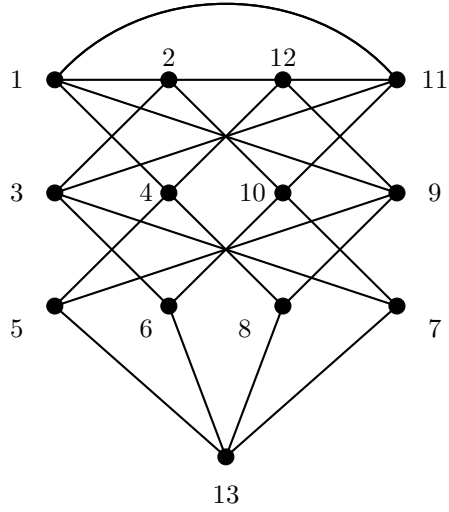
\begin{lemma} \label{mtw41}
$\mu_m(W_n)$, where $ W_n = C_n + K_1$ is not distance magic for $n (\ne 4) \ge 3$.
\end{lemma}
\begin{proof}
Let $\{x_1,\ x_2, \dots,x_n\}$ be the set of vertices lying on the rim of $W_n = C_n + K_1$ and $c_1$ be the central vertex, where $n \ge 3$. If $n = 3$ then $W_3 \cong K_4$ and by Corollary \ref{kn}, $\mu_m(W_3)$ is not distance magic. Next, we prove the case $n \ge 5$. On the contrary, suppose that Mycielskian of wheel $W_n = C_n + K_1$ is distance magic with distance magic labeling $f$. Then, $N_{G}(x_1) = \{x_2,\ x_n,\ c_1\}$ and $N_{G}(x_3) = \{ x_2,\ x_4,\ c_1\}$, where the subscripts are taken to be modulo $n$. Then, $|N_{G}(x_1) \triangle N_{G}(x_3)| =|\{x_4,\ x_n\}| = 2$. Therefore, by Lemma \ref{VJ for Mycielski}, generalised Mycielskian of wheel $W_n = C_n + K_1$ for $n (\ne 4) \ge 3$ is not distance magic.
\end{proof}
The fractional total domination number of $W_4$ is known to be $\gamma_{ft}(W_4) = \frac{3}{2}$. Referring to Theorem \ref{dom_num}, we can determine that
\begin{equation*}
\gamma_{ft}(\mu_m(W_4)) = 
\begin{cases}
\frac{9m + 16}{12} & \text{if } m \text{ is even}\\
\frac{9m + 17}{12} & \text{if } m \text{ is odd}
\end{cases}
\end{equation*}
Furthermore, according to Theorem \ref{unique_k}, if $\mu_m(W_4)$ is a distance magic graph, its magic constant $k$ is given by
\begin{equation*}
k = 
\begin{cases}
\frac{6(5m + 6) (5m + 7)}{9m + 16} & \text{if } m \text{ is even}\\
\frac{6(5m + 6) (5m + 7)}{9m + 17} & \text{if } m \text{ is odd}
\end{cases}
\end{equation*}
It should be noted that $k$ may not always be an integer. Specifically, for $m = 2$ and $m = 4$, we obtain integer values $k = 48$ and $k = 81$ respectively, but none of the $\mu_2(W_{4})$ or $\mu_{4}(W_4)$ is distance magic as discussed below.
\begin{lemma} \label{mtw42}
The generalized Mycielskian of wheel $\mu_m(W_4)$ is not distance magic.
\end{lemma}
\begin{proof}
On the contrary, suppose that $\mu_m(W_4)$ is distance magic and has distance magic labeling $f$. Let $x_{1}x_{2}x_{3}x_{4}x_{1}$ forms a cycle lying on the rim of $W_4$ and $c_1$ be the central vertex. Then the neighborhood of the vertices is as follows:
\begin{align*}
& N(x_1,0) = N(x_3,0) = \{(x_2,0), (x_4,0), (x_2,1), (x_4,1), (c,0), (c,1)\}\\
& N(x_2,0) = N(x_4,0) = \{(x_1,0), (x_3,0), (x_1,1), (x_3,1), (c,0), (c,1)\}\\
& N(x_1, i) = N(x_3,i) = \{(x_2,i-1), (x_4,i-1), (x_2,i+1), (x_4,i+1), (c,i-1), (c,i+1)\}, \mbox{where}\;  1 \le i \le m-1\\
& N(x_2, i) = N(x_4,i) = \{(x_1,i-1), (x_3,i-1), (x_1,i+1), (x_3,i+1), (c,i-1), (c,i+1)\}, \mbox{where} \;  1 \le i \le m-1\\
& N(x_1,m) = N(x_3,m) = \{(x_2,m-1), (x_4,m-1), (c,m-1), u\}\\
& N(x_2,m) = N(x_4,m) = \{(x_1,m-1), (x_3,m-1), (c,m-1), u\}\\
& N(u) = \{(x_1,m), (x_2,m), (x_3,m), (x_4,m), (c,m)\}\\
& N(c,0) = \{(x_1,0), (x_2,0), (x_3,0), (x_4,0), (x_1,1), (x_2,1), (x_3,1), (x_4,1)\}\\
& N(c,i) = \{(x_1,i-1), (x_2,i-1), (x_3,i-1), (x_4,i-1), (x_1,i+1), (x_2,i+1), (x_3,i+1), (x_4,i+1)\}, \mbox{where}\;  1 \le i \le m-1\\
& N(c,m) = \{(x_1,m-1), (x_2,m-1), (x_3,m-1), (x_4,m-1), u\}.
\end{align*}
		
We first prove the result for $m = 2$. The generalised Mycielskian of $W_4$ with $m = 2$ is shown in Figure \ref{fig:my.wheel}. We assume that,
\begin{align*}
& \alpha_i = f(x_1,i) + f(x_3,i) \text{ for } i = 0,1,2\\
& \beta_i = f(x_2,i) + f(x_4,i) \text{ for } i = 0,1,2.  
\end{align*}
		
With the above notations, we express the weights of vertices in $\mu_m(W_4)$ under $f$ as follows:  
		\begin{align}
			& w(x_1,0) = \beta_0 + \beta_1 + f(c,0) + f(c,1) \label{eq:x10}\\
			& w(x_2,0) = \alpha_0 + \alpha_1 + f(c,0) + f(c,1) \label{eq:x20}\\
			& w(x_1,1) = \beta_0 + \beta_2 +  f(c,0) + f(c,2) \label{eq:x11}\\
			& w(x_2,1) = \alpha_0 + \alpha_2 + f(c,0) + f(c,2) \label{eq:x21}\\
			& w(x_1,2) = \beta_1 + f(c,1) + f(u) \label{eq:x12}\\
			& w(x_2,2) = \alpha_1 + f(c,1) + f(u) \label{eq:x22}\\
			& w(c,0) = \alpha_0 + \beta_0 + \alpha_1 + \beta_1 \label{eq:c0}\\
			& w(c,1) = \alpha_0 + \beta_0 + \alpha_2 + \beta_2 \label{eq:c1}\\
			& w(c,2) = \alpha_1 + \beta_1 + f(u) \label{eq:c2}\\
			& w(u) = \alpha_2 + \beta_2 + f(c,2). \label{eq:u}
		\end{align}
Now we equate: Equation (\ref{eq:x10}) with Equation (\ref{eq:x12}), Equation (\ref{eq:x20}) with Equation (\ref{eq:x22}), Equation (\ref{eq:x11}) with Equation (\ref{eq:u}), Equation (\ref{eq:x21}) with Equation (\ref{eq:u}), and Equation (\ref{eq:c0}) with Equation (\ref{eq:c1}) respectively, to obtain, 
\begin{align}
			& f(u) = \beta_0 + f(c,0) \label{eq:third2}\\
			& f(u) = \alpha_0 + f(c,0) \label{eq:fourth2}\\
			& \alpha_2 = \beta_0 + f(c,0) \label{eq:second2}\\
			& \beta_2 = \alpha_0 + f(c,0) \label{eq:fifth2}\\
			& \alpha_1 + \beta_1 = \alpha_2 + \beta_2 \label{eq:first2}.
\end{align}
Again from Equations (\ref{eq:fourth2}), (\ref{eq:third2}), (\ref{eq:second2}), (\ref{eq:fifth2}) we obtain $f(u) = \alpha_2 = \beta_2$. Using this value of $\alpha_2$ and $\beta_2$ in Equation (\ref{eq:first2}) we get $\alpha_1 + \beta_1 = 2f(u)$. Therefore Equation (\ref{eq:c2}) becomes $w(c,2) = 3f(u)$. Since $f$ is assumed to be distance magic labeling, we must have $w(u) = 3f(u)$. From Equation (\ref{eq:u}) we have $3f(u) = 2f(u) + f(c,2) \implies f(u) = f(c,2)$ which is a contradiction. Hence $\mu_m(W_4)$ is not distance magic for $m = 2$.
Now, we suppose that $m \ge 3$. By equating the weights: $w(x_1, 0) = w(x_1,2),\ w(x_2,0) = w(x_2,2)$ and $w(c,0) = w(c,2)$, respectively, we get 
\begin{align}
			&f(x_2,0) + f(x_4,0) + f(c,0) = f(x_2,3) + f(x_4,3) + f(c,3) \label{eq:w41}\\
			&f(x_1,0) + f(x_3,0) + f(c,0) = f(x_1,3) + f(x_3,3) + f(c,3) \label{eq:w42}\\
			&f(x_1,0) + f(x_2,0) + f(x_3,0) + f(x_4,0) = f(x_1,3) + f(x_2,3) + f(x_3,3) + f(x_4,3). \label{eq:w43}
		\end{align}
		Adding Equations (\ref{eq:w41}) and (\ref{eq:w42}) we get
		\begin{align*}
			& f(x_1,0) + f(x_2,0) + f(x_3,0) + f(x_4,0) + 2f(c,0) = f(x_1,3) + f(x_2,3) + f(x_3,3) + f(x_4,3) + 2f(c,3)\\
			& \implies 2f(c,0) = 2f(c,3) \text{ \dots using Equation (\ref{eq:w43})}\\
			& \implies f(c,0) = f(c,3).
\end{align*}
Which is a contradiction. This completes the proof.
\end{proof}
\begin{figure}[ht]
		\centering
		\begin{tikzpicture}
			\draw[fill=black] (0,0) circle (3pt);
			\draw[fill=black] (1.5,0) circle (3pt);
			\draw[fill=black] (3,0) circle (3pt);
			\draw[fill=black] (4.5,0) circle (3pt);
			\draw[fill=black] (6.5,0) circle (3pt);
			
			\draw[fill=black] (0,-1.5) circle (3pt);
			\draw[fill=black] (1.5,-1.5) circle (3pt);
			\draw[fill=black] (3,-1.5) circle (3pt);
			\draw[fill=black] (4.5,-1.5) circle (3pt);
			\draw[fill=black] (6.5,-1.5) circle (3pt);
			
			\draw[fill=black] (0,-3) circle (3pt);
			\draw[fill=black] (1.5,-3) circle (3pt);
			\draw[fill=black] (3,-3) circle (3pt);
			\draw[fill=black] (4.5,-3) circle (3pt);
			\draw[fill=black] (6.5,-3) circle (3pt);
			
			\draw[fill=black] (3.25,-5) circle (3pt);
			
			\draw [thick] (0,0) -- (1.5,0) -- (3,0) -- (4.5,0);
			\draw[thick] (0,0) to [bend left=50] (4.5,0);
			\draw[thick] (4.5,0) to [bend right=50] (0,0);
			
			\draw[thick] (0,0) -- (1.5,-1.5) -- (3,0) -- (4.5,-1.5) -- (0,0);
			\draw[thick] (0,-1.5) -- (1.5,0) -- (3,-1.5) -- (4.5, 0) -- (0,-1.5);
			\draw[thick] (0,-3) -- (1.5,-1.5) -- (3,-3) -- (4.5,-1.5) -- (0,-3);
			\draw[thick] (0,-1.5) -- (1.5,-3) -- (3,-1.5) -- (4.5,-3) -- (0,-1.5);
			
			\draw[thick] (0,0) to [bend left=100] (6.5,0);
			\draw[thick] (6.5,0) to [bend right=100] (0,0);
			\draw[thick] (1.5,0) to [bend left=90] (6.5,0);
			\draw[thick] (6.5,0) to [bend right=90] (1.5,0);
			\draw[thick] (3,0) to [bend left=80] (6.5,0);
			\draw[thick] (6.5,0) to [bend right=80] (3,0);
			\draw[thick] (4.5,0) -- (6.5,0);
			
			\draw[thick] (0,0) -- (6.5,-1.5) -- (1.5,0);
			\draw[thick] (3,0) -- (6.5,-1.5) -- (4.5,0);
			
			\draw[thick] (0,-1.5) -- (6.5,0) -- (1.5,-1.5);
			\draw[thick] (3,-1.5) -- (6.5,0) -- (4.5,-1.5);
			
			\draw[thick] (0,-1.5) -- (6.5,-3) -- (1.5,-1.5);
			\draw[thick] (3,-1.5) -- (6.5,-3) -- (4.5,-1.5);
			
			\draw[thick] (0,-3) -- (6.5,-1.5) -- (1.5,-3);
			\draw[thick] (3,-3) -- (6.5,-1.5) -- (4.5,-3);
			
			\draw[thick] (0,-3) -- (3.25,-5) -- (1.5,-3);
			\draw[thick] (3,-3) -- (3.25,-5) -- (4.5,-3);
			\draw[thick] (6.5,-3) -- (3.25,-5);
			
			\node at (3.25,-5.3) {$u$};
			\node at (-0.5,-3.3) {$(x_1,2)$};
			\node at (1.2,-3.3) {$(x_2,2)$};
			\node at (2.5,-3.3) {$(x_3,2)$};
			\node at (5,-3.3) {$(x_4,2)$};
			\node at (7,-3.3) {$(c,2)$};
			
			\node at (-0.7,0) {$(x_1,0)$};
			\node at (-0.7,-1.5) {$(x_1,1)$};
			\node at (7,-1.5) {$(c,1)$};
			\node at (7,0) {$(c,0)$};
			\node at (1,0.2) {$(x_2,0)$};
			\node at (2.5,0.2) {$(x_3,0)$};
			\node at (5.3,0.2) {$(x_4,0)$};
			
		\end{tikzpicture}
		\caption{$\mu_2(W_4)$.}
		\label{fig:my.wheel}
\end{figure}
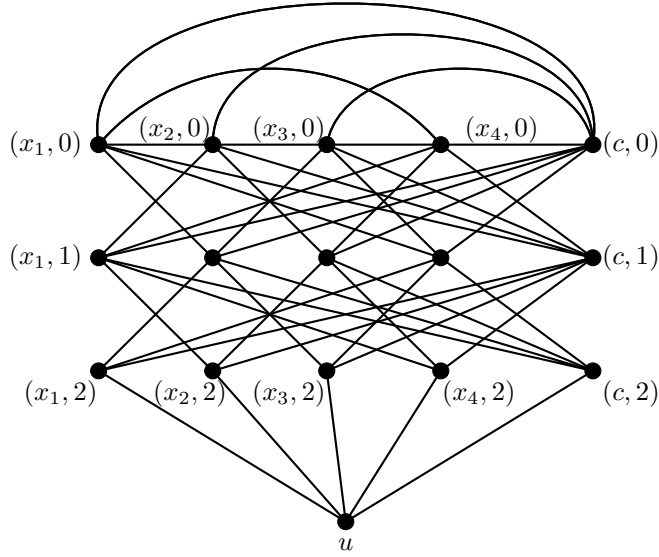
From the Lemma \ref{mtw41} and the Lemma \ref{mtw42}, the following result is evident.
\begin{theorem}
$\mu_m(W_n)$ is not distance magic for $n \ge 3$.
\end{theorem}
Let us calculate the fractional total domination number of generalised Mycielskian of complete bipartite graphs. We know that $\gamma_{ft}(K_{a,b}) = 2$. Therefore by Theorem \ref{dom_num}, $\gamma_{ft}(\mu_{m}(K_{a,b})) = \frac{2m+3}{2}$ for all $m \ge 1$. If $\mu_{m}(K_{a,b})$ is distance magic then by Theorem \ref{unique_k}, its magic constant is given by $k = \frac{(m(a+b) + a+b+1)(m(a+b) + a+b+2)}{2m+3}$. When $a=b=1$, then $k = 2m+4$ which is an integer but $K_{1,1} \cong K_2$ and by Lemma \ref{deg1}, $\mu_m(K_2)$ is not distance magic. For the case where $a = b = 2$ and $m \ge 1$, the value of $k$ evaluates to an integer $8m + 10$. However, for other values of $a, b$, and $m$, the calculated value of $k$ may not necessarily be an integer and we fail to conclude anything about its distance magic property. In the following theorem, we give a simple proof for the distance magic labeling of complete bipartite graphs.\\
\smallskip
\begin{theorem} \label{bipartite}
$\mu_m(K_{a,b})$ is distance magic if and only if $a = b = 2$.
\end{theorem}
\begin{proof}
Let $G \cong K_{a,b}$, where $a$ and $b$ both are at least $2$. Otherwise, $G$ will be a star, and by Lemma \ref{deg1}, it is not distance magic. Let $V_1 = \{ x_1, x_{2}, \dots, x_{a} \} $ and $ V_2 = \{ y_{1}, y_{2}, \dots, y_{b} \} $ be the partition of vertex set of $G$. Suppose that $\mu_m(G)$ is distance magic with distance magic labeling $f$. We assume that
\begin{equation*}
\sum_{i=1}^a f(x_i,j) = \alpha_{j} \text{ for } j=0,1,2, \dots, m \text{ and }
\sum_{i=1}^b f(y_i,j) = \beta_{j} \text{ for } j=0,1,2, \dots, m.
\end{equation*}
Therefore, the vertex weights are as follows:
\begin{align*}
			& w(x_{i},0) = \beta_{0} + \beta_{1}, \text{ for } i = 1,2, \dots, a\\ 
			& w(x_{i}, j) = \beta_{j-1} + \beta_{j+1}, \text{ for } i = 1,2, \dots, a; j=1,2, \dots m-1\\
			& w(x_{i}, m) = \beta_{m-1} + f(u), \text{ for } i = 1,2, \dots, a\\
			& w(y_{i},0) = \alpha_{0} + \alpha_{1}, \text{ for } i = 1,2, \dots, b\\ 
			& w(y_{i}, j) = \alpha_{j-1} + \alpha_{j+1}, \text{ for } i = 1,2, \dots, b; \; j=1,2, \dots m-1\\
			& w(y_{i}, m) = \alpha_{m-1} + f(u), \text{ for } i = 1,2, \dots, b.
\end{align*}
Since $\mu_m(G)$ is distance magic, the vertex weights are the same under $f$. Equating the weights we get, 
\begin{eqnarray}\label{eq:bipartite1}
\alpha_{i} = \beta_{i} = f(u).
\end{eqnarray}
The vertex $u$ must receive the largest label i.e. $f(u) = (a+b)(m+1)+1$. Otherwise, if one of the vertex $(x_i,j)$ or $(y_i,j)$ receives the label $(a+b)(m+1)+1$ for some $i$ and $j$ then one of the equality 
\begin{equation*}
\alpha_i = f(u),\ \beta_i = f(u)
\end{equation*} is not possible.
Therefore, from equation (\ref{eq:bipartite1}) we have 
\begin{equation}\label{eq:bipartite2}
\sum_{i=0}^{m} \left(\alpha_i + \beta_i \right) = 2(m+1) f(u).
\end{equation} 
Since, $f(u) = (a+b)(m+1)+1$. Therefore, $\sum_{i=0}^{m} \left(\alpha_i + \beta_i \right)$ is the sum of the first $(a+b)(m+1)$ natural numbers. Therefore, equation (\ref{eq:bipartite2}) becomes   
\begin{equation}\nonumber
\frac{(a+b)(m+1)[(a+b)(m+1)+1]}{2} = 2(m+1) [(a+b)(m+1)+1].
\end{equation}
This implies $a+b = 4$.
Since, $a$ and $b$ both are at least $2$, we must have $a = b = 2$.
Conversely, suppose that $a=b=2$. In this case $K_{2,2} \cong C_4$ and by Theorem \ref{cn}, $\mu_m(C_4)$ is distance magic. This completes the proof.
\end{proof}

Let $\mu_m(G)$ be distance magic generalised Mycielskian graph with distance magic labeling $f$ and magic constant $k$. For any vertex $x \in V(G)$, the sum of labels $\sum_{y \in N(x)} f(y,j)$ alternates in $f(u)$ and $k - f(u)$.

\begin{theorem} \label{th: alt}
Let $\mu_m(G)$ admits distance magic labeling $f$ with magic constant $k$. Then for a given $x \in V(G)$, $\sum_{y \in N(x)} f(y,j) = f(u)$ or $k-f(u)$ for each $j = 0,1,2, \dots, m$.
\end{theorem}
\begin{proof}
Let $\mu_m(G)$ admits distance magic labeling $f$ with magic constant $k$. Let $x \in V(G)$ be arbitrary. We consider the following system of equations:
\begin{align*}
			&w(x,0) = \sum_{y \in N(x)} f(y,0) + \sum_{y \in N(x)}f(y,1)\\
			&w(x,1) = \sum_{y \in N(x)} f(y,0) + \sum_{y \in N(x)}f(y,2)\\
			&w(x,2) = \sum_{y \in N(x)} f(y,1) + \sum_{y \in N(x)}f(y,3)\\
			&\dots\\
			&w(x,m-1) = \sum_{y \in N(x)} f(y,m-2) + \sum_{y \in N(x)}f(y,m)\\
			&w(x,m) = \sum_{y \in N(x)} f(y,m-1) + f(u).
\end{align*}
For $m = 2$, it is easy to observe that $\sum_{y \in N(x)} f(y,0) = k$ and $\sum_{y \in N(x)}f(y,1) = k - f(u)$. Also for $m = 3$, $\sum_{y \in N(x)}f(y,1) = \sum_{y \in N(x)}f(y,2) = k - f(u)$ and $\sum_{y \in N(x)}f(y,0) = k$.\\
Let $m \ge 4$. For a given $j (1 \le j \le m-3)$, we compare weights of the vertices $(x, j)$ and $(x, j+2)$ to obtain
\begin{equation} \label{eq: rel1}
\sum_{y \in N(x)}f(y, j-1) = \sum_{y \in N(x)}f(y, j+3).
\end{equation}
Consider the following cases:\\

\noindent {\bf Case-i:} If $j \equiv 0 \pmod{4}$ then using Equation \ref{eq: rel1}, we obtain the following equalities
\begin{equation} \nonumber
\sum_{y \in N(x)}f(y, 3) = \sum_{y \in N(x)}f(y, 7) = \sum_{y \in N(x)}f(y, 11) = \cdots
\end{equation}
\noindent {\bf Case-ii:} If $j \equiv 1 \pmod{4}$ then using Equation \ref{eq: rel1}, we obtain the following equalities
\begin{equation} \nonumber
\sum_{y \in N(x)}f(y, 0) = \sum_{y \in N(x)}f(y, 4) = \sum_{y \in N(x)}f(y, 8) = \cdots
\end{equation}
\noindent {\bf Case-iii:}  If $j \equiv 2 \pmod{4}$ then using Equation \ref{eq: rel1}, we obtain the following equalities
\begin{equation} \nonumber
\sum_{y \in N(x)}f(y, 1) = \sum_{y \in N(x)}f(y, 5) = \sum_{y \in N(x)}f(y, 9) = \cdots
\end{equation}
\noindent {\bf Case-iv:} If $j \equiv 0 \pmod{4}$ then using Equation \ref{eq: rel1}, we obtain the following equalities
\begin{equation} \nonumber
\sum_{y \in N(x)}f(y, 2) = \sum_{y \in N(x)}f(y, 6) = \sum_{y \in N(x)}f(y, 10) = \cdots
\end{equation}
Equating $w(x, 0) = w(x, 1)$ we obtain, $\sum_{y \in N(x)}f(y, 1) = \sum_{y \in N(x)}f(y, 2)$ and equating $w(x,0) = w(x,2)$ we obtain, $\sum_{y \in N(x)}f(y, 0) = \sum_{y \in N(x)}f(y, 3)$. Using these two equalities with the above four cases, we obtain the following equalities
\begin{align}
&\sum_{y \in N(x)}f(y, 0) = \sum_{y \in N(x)}f(y, 3) = \sum_{y \in N(x)}f(y, 4) = \sum_{y \in N(x)}f(y, 7) = \sum_{y \in N(x)}f(y, 8) = \cdots \intertext{and}
&\sum_{y \in N(x)}f(y, 1) = \sum_{y \in N(x)}f(y, 2) = \sum_{y \in N(x)}f(y, 5) = \sum_{y \in N(x)}f(y, 6) = \cdots.
\end{align}
		
Hence the $m+1$ sums $\sum_{y \in N(x)} f(y,j)$ for $j = 0,1,2 \dots, m$ can be partitioned into two sets $S_1$ and $S_2$ according to $j \equiv 0 \text{ or } 3 (\bmod\ 4)$ and $j \equiv 1 \text{ or } 2(\bmod\ 4)$, respectively giving constant sums. 
We will show that if sum of all elements in $S_1$ is $f(u)$ then sum of all elements in $S_2$ is $k-f(u)$ or vice versa.\\

Equating $w(x,m) = w(x, m-2)$, we obtain, $f(u) = \sum_{y \in N(x)}f(y, m-3)$. Observe that when $m$ is odd, $\sum_{y \in N(x)}f(y, m) = \sum_{y \in N(x)}f(y, m-3) = f(u)$. We have $k = w(x, m) = \sum_{y \in N(x)}f(y, m-1) + f(u)$ which gives $\sum_{y \in N(x)}f(y, m-1) = k - f(u)$. If $m \equiv 0(\bmod\ 4)$ then $m-1 \equiv 3(\bmod\ 4)$. Hence by obtained partitioning of sums, $\sum_{y \in N(x)}f(y, m) = \sum_{y \in N(x)}f(y, m - 1) = k - f(u)$. Similarly, when $m \equiv 2(\bmod\ 4)$ we can show that $\sum_{y \in N(x)}f(y, m) = k - f(u)$. Hence we have the following:
\begin{equation*}
\sum_{y \in N(x)} f(y,m) = 
\begin{cases}
f(u), \text{ if } m \text{ is odd}\\
k-f(u), \text{ if } m \text{ is even}.
\end{cases}
\end{equation*}
Now, for any given $m$, $\sum_{y \in N(x)}f(y, m)$ is in exactly one of the sets $S_1$ or $S_2$. This shows that sums in each of the partitions are constant. This proves the theorem.
\end{proof}
\begin{theorem} \label{th: regular}
Let $G$ be an $r$-regular graph and  $m$ be odd. If $\mu_m(G)$ is distance magic then $r < 2(m+1)$.
\end{theorem}
\begin{proof}
Let $G$ be an $r$-regular graph on $n$ vertices such that $\mu_m(G)$ admits distance magic labeling $f$. Then $deg_{\mu_m(G)}(x,i) = 2r$, for $(0 \le i \le m-1)$, $deg_{\mu_m(G)}(x,m) = r+1$ and $deg_{\mu_m(G)}(u) = n$. Now by Theorem \ref{weights}, we have
\begin{align*}\label{reg1}
& \sum_{x \in V(G)} w(x,0) = r \sum_{x \in V(G)} f(x,0) + r \sum_{x \in V(G)} f(x,1) = kn\\   
& \sum_{x \in V(G)} w(x,1) = r \sum_{x \in V(G)} f(x,0) + r \sum_{x \in V(G)} f(x,2) = kn\\
& \dots\\
& \sum_{x \in V(G)} w(x,m-1) = r \sum_{x \in V(G)} f(x,m-2) + r \sum_{x \in V(G)} f(x,m) = kn\\
& \sum_{x \in V(G)} w(x,m) = r \sum_{x \in V(G)} f(x,m-1) + nf(u) = kn.
\end{align*}
Observe that when $m$ is odd,
\begin{equation} \label{eq:regular2}
 r \sum_{x \in V(G)} f(x,m) = n f(u).
\end{equation}
If we assign the smallest labels $1,2,3,\dots,n$ to the $n$ vertices $(x,m)$ and label $u$ with the largest label i.e. $f(u) = mn+n+1$, then using this in the Equation (\ref{eq:regular2}), we get $n (mn+n+1) \geq \frac{rn(n+1)}{2}$ which on simplification yields $r < 2(m+1)$.
\end{proof}
	
\begin{proposition}\label{mreg}
Let $G$ be a graph. If $\mu_m(G)$ is a regular graph, then $\mu_m(G)$ is not distance magic.
\end{proposition}
\begin{proof}
Let $G$ be a graph on $n$ vertices such that $\mu_m(G)$ is $r$ regular. Then by Theorem \ref{th2}, $G \cong K_2$. But $\mu_m(K_2) \cong C_{2m+3}$ and by Theorem \ref{cycle}, $C_{2m+3}$ is not distance magic. This completes the proof.
\end{proof}
	
\begin{observation} \label{ob1}
The graph $G$ and its generalised Mycielskian $\mu_m(G)$ do not share the property of being distance magic, i.e. $\mu_m(G)$ is distance magic irrespective of $G$, e.g., the path on $3$ vertices $P_3$ is distance magic \cite{dm_miller} but $\mu_m(P_3)$, $(m \ge 2)$ is not distance magic (see Corollary \ref{cordeg1}). Whereas $C_4$ is distance magic \cite{dm_miller} and $\mu_m(C_4)$, $(m \ge 2)$ is also distance magic (see Theorem \ref{cn}).
\end{observation}
	
\section{Conclusion and Future Scope}
The property of a graph being distance magic is not preserved under the Mycielskian construction for all graphs. Finding classes of graphs for which the property of a graph being distance magic is preserved under the repetitive application of the generalized Mycielskian construction will lead to a class of distance magic graphs with arbitrarily large chromatic numbers. The complete characterization of distance magic labeling of the generalized Mycielskian graph is an open problem.

\bibliography{refs}
\bibliographystyle{plain}

\end{document}